\documentclass[12pt]{amsart}
\usepackage{amsthm,amssymb,amsmath,enumerate,graphicx, tikz}
\usepackage[margin=1.5in]{geometry}
\usepackage{amscd}
\usetikzlibrary{decorations.markings}
\usepackage[justification=centering]{caption}
\usepackage{setspace}
\usepackage{comment}
\usepackage{mathtools}
\usepackage{tikz}

\newcommand{\boundellipse}[3]
{(#1) ellipse (#2 and #3)
}
\usetikzlibrary{decorations.pathreplacing}

\theoremstyle{plain}
\newtheorem{theorem}{Theorem}[section]

\newtheorem{lemma}[theorem]{Lemma}
\newtheorem{observation}[theorem]{Observation}

\newtheorem{conjecture}[theorem]{Conjecture}

\theoremstyle{definition}
\newtheorem{definition}[theorem]{Definition}

\newtheorem{question}[theorem]{Question}

\newtheorem{example}[theorem]{Example}

\theoremstyle{remark}

\newcommand{\F}{\mathcal{F}}

\newcommand{\Forb}{\text{Forb}}

\newcommand{\Tr}[2]{T_{#1,#2}}
\newcommand{\cTr}[2]{\overline{T}_{#1,#2}}

\title{A Proof of the $(n,k,t)$-Conjectures}

\author{Stacie Baumann}
\address{Department of Mathematics, College of Charleston\\
SC, USA}
\email{baumannsm@cofc.edu}

\author{Joseph Briggs}
\address{Department of Mathematics and Statistics,
Auburn University\\
AL, USA}\email{jgb0059@auburn.edu}

\begin{document}

\maketitle
\begin{abstract}
An \emph{$(n,k,t)$-graph} is a graph on $n$ vertices in which every 
set of $k$ vertices contains a clique on $t$ vertices.
Tur\'an's Theorem, rephrased in terms of graph complements, states that the unique minimum $(n,k,2)$-graph is an equitable disjoint union of cliques.
We prove that minimum $(n,k,t)$-graphs are always disjoint unions of cliques for any $t$ (despite \allowbreak nonuniqueness of extremal examples), thereby generalizing
Tur\'an's Theorem and
confirming two conjectures of Hoffman et al.



\end{abstract}

\section{Introduction}


The protagonists of this paper are  $(n,k,t)$-graphs.
Throughout this paper, it is assumed $n,k,t$  and $r$ are  positive integers and $n \geq k \geq t$.

\begin{definition}
A graph $G$ is an $(n,k,t)$-\emph{graph} if $\vert V(G) \vert = n$ and every induced subgraph on $k$ vertices contains a clique on $t$ vertices.
A \emph{minimum} $(n,k,t)$-graph is an $(n,k,t)$-graph with the minimum number of edges among all $(n,k,t)$-graphs.
\end{definition}

The study of the minimum number of edges in an $(n,k,t)$-graph, and so implicitly the structure of minimum $(n,k,t)$-graphs, is a natural extremal graph theory problem---as we will see shortly, this setting generalizes the flagship theorems of Mantel and Tur\'an.
Hoffman, Johnson, and McDonald \cite{Hoffman} conjectured the following.
 
\begin{conjecture}[The Weak $(n,k,t)$-Conjecture] \label{statementweakconjecture}
For any positive integers $n \geq k \geq t$, there exists a minimum $(n,k,t)$-graph that is a disjoint union of cliques.
\end{conjecture}

\begin{conjecture}[The Strong $(n,k,t)$-Conjecture] \label{strongconj}
For any postive integers $n \geq k \geq t$,
every minimum $(n,k,t)$-graph is a disjoint union of cliques.
\end{conjecture}

Note that throughout this paper ``disjoint union'' refers to a \emph{vertex}-disjoint union, and a disjoint union of cliques allows isolated vertices (namely cliques of size 1).

Partial progress toward proving the conjectures can be found in \cite{Hoffman, noble2017application}.
These results will be discussed in more detail in Section \ref{prevresults}.
We prove the following theorem, confirming the Strong (and therefore also the Weak) $(n,k,t)$-Conjecture.  

\begin{theorem} \label{strongthm}
For any positive integers $n \geq k \geq t$, all minimum $(n,k,t)$-graphs are a disjoint union of cliques.
\end{theorem}

We prove Theorem \ref{strongthm} by proving a stronger statement involving the independence number of the graph (see Theorem \ref{mainthm}).

Given the implicit equivalence relation on the vertex set arising from the adjacency relation in a disjoint union of cliques, a natural approach to try to prove Conjecture \ref{statementweakconjecture} is to mimic the Zykov symmetrization proof \cite{zykov1949some} of Tur\'{a}n’s theorem, which (when complemented) involves repeated applications of the following:
\begin{observation}
    Let G be an $(n,k,2)$-graph and $x$,$y$ be adjacent vertices in G. Then replacing $y$ with a copy $x’$ of $x$, with $x’$ adjacent to $x$, also gives an $(n,k,2)$-graph.
    \end{observation}

In contrast, there is no equivalent fact for $(n,k,t)$-graphs when $t>2$. For example, a $K_4$ with a pendant edge $xy$ (such that $y$ is in the $K_4$) is a $(5,4,3)$-graph, but replacing $y$ with a copy $x’$ of $x$ gives a disjoint triangle and edge, which is not a $(5,4,3)$-graph.

\section{Previous Results} \label{prevresults}

Given graphs $G$ and $H$, let $\overline{G}$ denote the complement of $G$, let $G + H$ denote the disjoint union of $G$ and $H$, and for a positive integer $s$ let $sG = G + \dots + G$ ($s$ times).
The authors in \cite{Hoffman} discussed the following basic cases of the Strong $(n,k,t)$-Conjecture.
\begin{itemize}
    \item $t=1$. Every graph on $n$ vertices is an $(n,k,1)$-graph, so the unique minimum $(n,k,1)$-graph is $nK_1$.
    \item $k=t \geq 2$. The unique minimum $(n,k,k)$-graph is $K_n$.
    \item $n=k$. The unique minimum $(n,n,t)$-graph is $(n-t)K_1+K_t$.
\end{itemize}

When $t=2$, Tur\'{a}n's Theorem, which we recall here, implies the Strong $(n,k,t)$-conjecture.
If $r \leq n$ are positive integers let $\Tr{n}{r} = \overline{K_{p_1} + \dots + K_{p_r}}$ where $p_1 \leq p_2 \leq \dots \leq p_r \leq p_1+1$ and $p_1 + p_2 \dots + p_r =n$ ($\Tr{n}{r}$ is called the \emph{Tur\'{a}n Graph}).
That is to say, $\cTr{n}{r}$ is a disjoint union of $r$ cliques where the number of vertices in each clique differs by at most one.
Tur\'{a}n's Theorem will be used throughout this paper.
The traditional statement, when rephrased in terms of graph complements, is below.

\begin{theorem}[Tur\'{a}n's Theorem]\label{turan}
The unique graph on $n$ vertices without an independent set of size $r+1$ with the minimum number of edges is $\cTr{n}{r}$.
\end{theorem}

By Tur\'{a}n's Theorem, the unique minimum $(n,k,2)$-graph is $\cTr{n}{k-1}$.
Indeed, the fact that Tur\'{a}n's Theorem (this version) depends on the independence number of a graph inspired the proof of the main theorem.

Arguably the most famous research direction in extremal graph theory is study of Tur\'an numbers, that is, for a fixed graph $F$ and the class $\Forb_n(F)$ of $n$-vertex graphs not containing a copy of $F$, what is the maximum number of edges? Tur\'an's Theorem gives an exact answer when $F$ is complete. Denoting $\chi(F)$ for the chromatic number of $F$, the classical Erd\H{o}s-Stone Theorem \cite{erdos1946structure} gives an asymptotic answer of $(1-1/(\chi(F)-1)) \binom{n}{2}$ for nonbipartite $F$ as $n \rightarrow \infty$, and the bipartite $F$ case is still a very active area of research (see e.g. \cite{bukh2015random}, \cite{furedi2013history}). A more general question considers a family $\F$ of graphs and the collection $\Forb_n(\F)$ of all $n$-vertex graphs not containing any subgraph in $\F$. In this language, an $(n,k,t)$-graph is precisely the complement of a graph in $\Forb_n(\F)$, where $\F$ is the family of $k$-vertex graphs $F$ where $\bar{F}$ does not contain $K_t$. Note that this $\F$ almost always has $\max_{F\in \F} \{\chi(F)\} >2$, meaning the Erd\H{o}s-Stone Theorem determines the correct edge density for large $n$. But Theorem \ref{mainthm} is still interesting for 2 reasons:
\begin{itemize}
    \item The values of $k$ and $t$ are arbitrary, not fixed (and may grow with $n$), and
    \item This result is exact, not asymptotic.
\end{itemize}

There has been some previous work towards the proof of the conjectures.
In \cite{noble2017application}, the Strong $(n,k,t)$-Conjecture was proved for $n \geq k \geq t \geq 3$ and $k \leq 2t-2$, utilizing
an extremal result about vertex covers attributed by Hajnal \cite{hajnal1965theorem} to Erd\H{o}s and Gallai \cite{erdos1961minimal}.
In \cite{Hoffman}, for all $n \geq k \geq t$,
the structure of
minimum $(n,k,t)$-graphs that are disjoint unions of cliques 
was described more precisely as follows.

\begin{theorem} [\cite{Hoffman}] \label{mincliquegraph}
Suppose $t \geq 2$ and $G$ has the minimum number of edges of all $(n,k,t)$-graphs that are a disjoint union of cliques.
Then $G=aK_1+ \cTr{n-a}{b}$, for some $a,b$ satisfying
$$
a+b(t-1)=k-1,
$$
and
$$
b \leq \min \left( \left\lfloor \frac{k-1}{t-1} \right\rfloor, n-k+1) \right).
$$
\end{theorem}

The following example shows Theorem \ref{strongthm} cannot be strengthened to include uniqueness (and in particular the choice of $b$ above need not be unique).

\begin{example} \label{multmin}
Theorem \ref{mincliquegraph} tells us the graphs with the minimum number of edges of all $(10,8,3)$-graphs that are a disjoint union of cliques are among $5K_1+\cTr{5}{1} = 5K_1+K_5$, $3K_1+\cTr{7}{2}=3K_1 + K_3 + K_4$, or $K_1+\cTr{9}{3}=K_1 + 3K_3$ (as given by $b=1,2,3$ in the above).
These graphs have $10$, $9$, and $9$ edges respectively.
So, $3K_1+\cTr{7}{2}=3K_1 + K_3 + K_4$ and $K_1+\cTr{9}{3} =K_1 + 3K_3$ both have the minimum number of edges of all $(10,8,3)$-graphs among disjoint unions of cliques.
\end{example}

Hoffman and Pavlis \cite{HoffPav} have found for any positive integer $N$ there exist some $n \geq k \geq 3$ with
at least $N$ non-isomorphic $(n,k,3)$-graphs which are minimum (among graphs which are disjoint unions of cliques, although Theorem \ref{mainthm} now removes this restriction).
Further, Allen et al. \cite{REU} determined precisely the values of $b$ from Theorem \ref{mincliquegraph} that minimize the number of edges.
The resulting growing families of nonisomorphic extremal constructions may suggest the difficulty of Conjectures \ref{statementweakconjecture} and \ref{strongconj}, but Observation \ref{uniqueindnum} puts such concerns to rest (and may be the reason why the problem still remains tractable).
Notice that the independence numbers of the graphs in Example \ref{multmin} are $6$, $5$, and $4$ respectively.
These lead to Observation \ref{uniqueindnum}, but first we state the following observation that will be used in the proof of Observation \ref{uniqueindnum} and frequently throughout the remainder of the paper.
Let $c(G)$ denote the number of connected components of a graph $G$, and let 
$\alpha(G)$ denote its independence number.

\begin{observation} \label{indnumcomponents}
Suppose $G$ is a disjoint union of cliques. Then $\alpha(G) =c(G)$.
\end{observation}

The following observation inspired considering the independence number in the main proof.

\begin{observation} \label{uniqueindnum}
Suppose $G_1$ and $G_2$ are non-isomorphic graphs that are both disjoint unions of cliques and are both minimum $(n,k,t)$-graphs. Then $\alpha(G_1) \neq \alpha(G_2)$.
\end{observation}

\begin{proof}
If $t=2$, then $G_1=G_2$ by uniqueness in Tur\'{a}n's Theorem.
So assume $t>2$.
For a contradiction, suppose $\alpha(G_1) = \alpha(G_2)$.
By Theorem \ref{mincliquegraph}, there exist positive integers $a_1, b_1, a_2, b_2$, such that $G_i = a_iK_1+ \cTr{n-a_{i}}{b_i}$ and $a_i+b_i(t-1)=k-1$ (for $i=1,2$).
Thus, $a_1+b_1(t-1)=a_2+b_2(t-1)$.
Also, by Observation \ref{indnumcomponents}, $a_1+b_1=a_2+b_2$.
Hence
$b_1-b_1(t-1)=b_2-b_2(t-1)$.
Combining these and noting $t > 2$ gives $b_1=b_2$
and $a_1=a_2$.
\end{proof}

\section{The Proof}

We collect together some preliminary lemmas. Firstly,
because we will use induction on $t$ in Theorem \ref{mainthm}, we observe that $(n,k,t)$-graphs contain $(n',k',t')$-graphs for certain smaller values of $n'$, $k'$, and $t'$.
Given a set $X \subseteq V(G)$, let $G[X]$ denote the subgraph of $G$ induced by the vertices in $X$ and let $G-X = G[V(G) \setminus X]$.

\begin{lemma} \label{minusset}
If $G$ is an $(n,k,t)$-graph and $S \subseteq V(G)$ is an independent set, 
then $G-S$ is an $(n- \vert S \vert, k - \vert S \vert, t-1)$-graph.
\end{lemma}

\begin{proof}
Clearly $\vert V(G-S) \vert = n - \vert S \vert$.
Let $X$ be a subset of $V(G-S)$ with $\vert X \vert = k- \vert S \vert$.
Because $G$ is an $(n,k,t)$-graph, $G[S \cup X]$ contains a $K_t$.
Because $S$ is an independent set, at most $1$ vertex in this $K_t$ was from $S$.
Thus, $(G-S)[X]$ must contain a $K_{t-1}$.
\end{proof}

Because Theorem \ref{mainthm} considers the independence number of $(n,k,t)$-graphs, we now 
enrich the $(n,k,t)$ notation to include the independence number.

\begin{definition}
A graph $G$ is an $(n,k,t,r)$-\emph{graph} if $G$ is an $(n,k,t)$-graph, and $\alpha(G)=r$.
A \emph{minimum} $(n,k,t,r)$-graph is an $(n,k,t,r)$-graph with the minimum number of edges among all $(n,k,t,r)$-graphs.
\end{definition}

The following lemma determines an upper bound for the independence number of an $(n,k,t)$-graph.

\begin{lemma} \label{maxalpha}
If $G$ is an $(n,k,t)$-graph, then $\alpha(G) < k-t+2$.
\end{lemma}

\begin{proof}
For a contradiction, suppose $G$ has an independent set, call it $S$, of size $k-t+2$.
Let $X'$ be any $t-2$ vertices in $G-S$ (this is possible because $\vert V(G-S) \vert = n- (k-t+2) = (n-k) + t-2 \geq t-2$).
Define $X \coloneqq S \cup X'$, so $\vert X \vert = k$.
Then there exists a $K_t$ in $G[X]$ which necessarily contains $2$ vertices in $S$.
This is a contradiction because $S$ is independent.
Thus $\alpha(G) < k-t+2$.
\end{proof}

But, subject to this bound, all independence numbers $\alpha(G)$ are attainable. 
Indeed, for any $1\leq r \leq k-t+1$ the graph $(r-1)K_1 + K_{n-r+1}$ is an example of an $(n,k,t,r)$-graph because every set of $k$ vertices contains at most $k-t$ isolated vertices and at least $t$ vertices of the clique.




Below is an observation about minimum $(n,k,t)$-graphs.
It will be used in in the proof of Theorem \ref{mainthm}.

\begin{lemma} \label{notkminus1}
Suppose $k>t$
and let $a \leq n$ be a positive integer.
Let $H$ be an $(n,k,t)$-graph with the minimum number of edges among $(n,k,t)$-graphs with independence number at most $a$.
If $\alpha(H) <a$, then $H$ is not an $(n,k-1,t)$-graph.
\end{lemma}

\begin{proof}
Among all $(n,k,t)$-graphs with independence number at most $a$, suppose $H$ has the fewest edges.
For a contradiction suppose $\alpha(H)<a$ and $H$ is an $(n,k-1,t)$-graph.
Because $\alpha(H)<a \leq n$, there exists an edge, $e$ in $H$.
Let $H^-$ be the graph formed from $H$ by deleting $e$ and let 
$v \in e$.
Let $X$ be a subset of $V(H^-)$ with $\vert X \vert = k$.
If $X$ does not contain $v$, then $H^-[X]=H[X]$ contains a $K_t$.
So, suppose $X$ contains $v$.
Then $\vert X \setminus \{v\} \vert = k-1$.
Because $H$ is an $(n, k-1, t)$-graph, $H[X \setminus \{v\}]=H^-[X \setminus \{v\}]$ contains a $K_t$.
So $H^-[X]$ contains a $K_t$.
Also, $H^-$ has independence number at most $1$ greater than $H$.
Thus, $H^-$ is an $(n,k,t)$-graph with fewer edges than $H$ and $\alpha(H^-) \leq a$, contradicting minimality of $H$.
\end{proof}

We caution that Lemma \ref{notkminus1} is not necessarily true when $\alpha(H)=a$.
For example, the minimum $(8,8,4)$-graph, $H$, with independence number $\alpha(H) \leq a:= 2$ is $K_4 + K_4$ because the minimum graph with independence number $2$ is $K_4+K_4$ by Mantel's Theorem (rephrased in terms of graph complements) and $K_4+K_4$ is an $(8,8,4)$-graph (because it contains $K_4$).
Despite this, $K_4 + K_4$ is also an $(8,7,4)$-graph by the pigeonhole principle.

Finally, since we will be using the $(n,k,t)$ condition to build the desired disjoint union of cliques, it is useful to keep track of the size of their largest $K_t$-free subgraph.

\begin{lemma}\label{l.nktrelation}
Suppose $\Gamma$ is a graph which is a disjoint union of cliques. Denote by $A_\Gamma$ the subgraph consisting of components with $<t$ vertices, and $B_\Gamma$ the subgraph of components with $\geq t$ vertices (so that $\Gamma=A_\Gamma + B_\Gamma$). 
Then:
\begin{enumerate}[a)]
\item the largest $K_t$-free subgraph of $\Gamma$ has $(t-1)c(B_\Gamma) +|V(A_\Gamma)|$ vertices, 
    \item if also $\Gamma$ is an $(n,k,t)$-graph, then \[
k-1 \geq (t-1) c(B_\Gamma) +|V(A_\Gamma)|, \hspace{5mm} \text{ and}
\]
\item if furthermore $\Gamma$ is \emph{not} an $(n,k-1,t)$-graph, the above inequality is an equality.
\end{enumerate}
\end{lemma}
\begin{proof}
The largest subgraph $F$ of $\Gamma$ with no $K_t$ is obtained by starting with $A_{\Gamma}$ and adding $t-1$ vertices from each clique of $B_{\Gamma}$. So $|V(F)|=(t-1)c(B_\Gamma)+|V(A_\Gamma)|$. If $\Gamma$ is an $(n,k,t)$-graph, then $F$ has at most $k-1$ vertices in total. If $\Gamma$ is not an $(n,k-1,t)$-graph, then there is some $(k-1)$-set of vertices without a $K_t$, and as $F$ is the largest, it follows $|V(F)| \geq k-1$ so we have equality.
\end{proof}

We now proceed with the proof of the main theorem.

\begin{theorem} \label{mainthm}

Every minimum $(n,k,t,r)$-graph is a disjoint union of cliques.
\end{theorem}

\begin{proof}
The proof proceeds by induction on $t$.
If $t=1$, then every graph on $n$ vertices is an $(n,k,1)$-graph because all sets of $k$ vertices contain a $K_1$.
Thus, the unique minimum $(n,k,1,r)$-graph is $\cTr{n}{r}$ by Tur\'{a}n's Theorem.
Else, $t \geq 2$.

Suppose $r \geq k-t+2$.
By Lemma \ref{maxalpha}, there does not exist an $(n,k,t,r)$-graph.
So the theorem holds vacuously.

Next suppose $r < \frac{k}{t-1}$.
By Tur\'{a}n's Theorem, $\cTr{n}{r}$ is the unique graph with independence number $r$ and the minimum number of edges.
So it suffices to check that $\cTr{n}{r}$ is genuinely an $(n,k,t)$-graph.
Let $X$ be a subset of $V(\cTr{n}{r})$ with $\vert X \vert = k$.
By the pigeonhole principle, $X$ contains at least $\frac{k}{r} > t-1$ vertices from one connected component of $\cTr{n}{r}$.
Because all connected components of $\cTr{n}{r}$ are cliques, $\cTr{n}{r}[X]$ contains a $K_t$.
Therefore $\cTr{n}{r}$ is the unique minimum $(n,k,t,r)$-graph.

So suppose instead $\frac{k}{t-1} \leq r < k-t+2$.
Let $G$ be an $(n,k,t,r)$-graph.
We first construct an $(n,k,t,r)$-graph $G'$ that is a disjoint union of cliques with $E(G) \geq E(G')$.
Let $S$ be an independent set in $G$ with $\vert S \vert = r$.
Then $\alpha(G-S) \leq \alpha(G) = r$.
By Lemma \ref{minusset}, $G-S$ must be an $(n-r, k-r, t-1)$-graph.
Let $H$ be any minimum $(n-r, k-r, t-1)$-graph with $\alpha(H) \leq r$,
so $\vert E(G-S) \vert \geq \vert E(H) \vert$.
By the induction hypothesis, $H$ is a disjoint union of cliques with $\alpha(H) \leq r$.
So, by Observation \ref{indnumcomponents}, $c(H) \leq r$.
Let $G'$ be the graph formed by increasing the size of each clique component of $G$ by one and adding $r-c(H)$ new isolated vertices
\begin{figure}[]
    \centering
    \begin{tabular}{cc}
    \begin{tikzpicture}[scale=0.6]
    
    \filldraw[] (0,4) circle [radius=0.125];
    \draw[very thick] (0,4) -- (-0.25,0.5);
    \draw[very thick] (0,4) -- (0.25,0.5);
    
    \filldraw[] (1,4) circle [radius=0.125];
    \draw[very thick] (1,4) -- (0.25,0.5);
    \draw[very thick] (1,4) -- (0.75,0.5);
    \draw[very thick] (1,4) -- (1.25,0.5);
    \draw[very thick] (1,4) -- (1.75,0.5);

    \filldraw[] (2,4) circle [radius=0.125];
    \draw[very thick] (2,4) -- (2,0.5);
    
    \filldraw[] (3,4) circle [radius=0.125];
    
    \filldraw[] (4,4) circle [radius=0.125];
    \draw[very thick] (4,4) -- (3.5,0.5);
    \draw[very thick] (4,4) -- (4.5,0.5);
    
    \filldraw[] (5,4) circle [radius=0.125];
    \draw[very thick] (5,4) -- (4.5,0.5);
    \draw[very thick] (5,4) -- (5.5,0.5);
    
    \filldraw[] (6,4) circle [radius=0.125];
    
    \filldraw[] (7,4) circle [radius=0.125];
    \draw[very thick] (7,4) -- (6.5,0.5);
    \draw[very thick] (7,4) -- (7,0.5);
    \draw[very thick] (7,4) -- (7.5,0.5);
    
    \filldraw[] (8,4) circle [radius=0.125];
    \draw[very thick] (8,4) -- (8,0.5);
    
    \draw [thick, decorate, 
    decoration = {brace, raise=2, amplitude=10}] (-0.25,4.4) --  (8.25,4.4);
    \node at (4,5.5) {$\vert S \vert =r$};
    
    \filldraw [gray!20, fill=gray!20] (-0.5,-3.9) rectangle (9,2);
    \node at (4,0) {Some};
    \node at (4,-0.75) {$(n-r, k-r, t-1)$-graph};
    
    \end{tikzpicture}
    
    &
    
    \begin{tikzpicture}[scale=0.6]
    \filldraw[] (0.25,4) circle [radius=0.125];
    \filldraw[] (0.75,4) circle [radius=0.125];
    \filldraw[] (1.25,4) circle [radius=0.125];
    \filldraw[] (1.75,4) circle [radius=0.125];
    \filldraw[] (2.25,4) circle [radius=0.125];
    \filldraw[] (3.25,4) circle [radius=0.125];
    \filldraw[] (3.75,4) circle [radius=0.125];
    \filldraw[] (5.6,4) circle [radius=0.125];
    \filldraw[] (7.4,4) circle [radius=0.125];
    
    \draw [thick, decorate, 
    decoration = {brace, raise=2, amplitude=10}] (3,4.4) --  (7.65,4.4);
    \node at (5.325,5.5) {$c(H)$};
    \draw [thick, decorate, 
    decoration = {brace, raise=2, amplitude=10}] (0,4.4) --  (2.5,4.4);
    \node at (1.6,5.5) {$r-c(H)$};
    
    \filldraw [gray!20, fill=gray!20] (-0.5,-3.9) rectangle (9,2);
    \node at (4,-3.2) {$H$};
    
    \draw[very thick] (3.25,4) -- (0.5,0.7);
    \draw[very thick] (3.25,4) -- (1,0.7);
    \draw[very thick] (3.75,4) -- (1.7,0.7);
    \draw[very thick] (3.75,4) -- (2.1,0.7);
    
    \draw[very thick] (5.6,4) -- (4.5,1.1);
    \draw[very thick] (5.6,4) -- (5,1.1);
    \draw[very thick] (5.6,4) -- (5.5,1.1);
    \draw[very thick] (7.4,4) -- (6.75,1.1);
    \draw[very thick] (7.4,4) -- (7.25,1.1);
    \draw[very thick] (7.4,4) -- (7.75,1.1);
    
    \draw[very thick, fill=gray!60] \boundellipse{0.75,0}{0.5}{0.8};
    \draw[very thick, fill=gray!60] \boundellipse{2,0}{0.5}{0.8};
    \draw[very thick, fill=gray!60]
    \boundellipse{5,0}{1}{1.25};
    \draw[very thick, fill=gray!60]
    \boundellipse{7.25,0}{1}{1.25};
    
    \draw[dashed] (0,1) -- (0,-1.75);
    \draw[dashed] (0,-1.75) -- (2.75,-1.75);
    \draw[dashed] (2.75,-1.75) -- (2.75,1);
    \draw[dashed] (0,1) -- (2.75,1);
    \node at (1.375,-1.25) {$A_H$};
    \draw[] (-0.25,1.45) -- (-0.25,-2.5);
    \draw[] (-0.25,-2.5) -- (3.15,-2.5);
    \draw[] (3.15,-2.5) -- (3.15,1.45);
    \node at (1.2,2.95) {$A_{G'}$};
    \draw[] (-0.25,1.45) -- (-0.25,3.75);
    \draw[] (-0.25,3.75) -- (-0.25, 4.25);
    \draw[] (-0.25, 4.25) -- (4.25, 4.25);
    \draw[] (4.25, 4.25) -- (4.25, 3.75);
    \draw[] (4.25, 3.75) -- (3.15,1.45);
    
    \draw[dashed] (3.75,1.45) -- (3.75,-1.75);
    \draw[dashed] (3.75,-1.75) -- (8.5,-1.75);
    \draw[dashed] (8.5,-1.75) -- (8.5,1.45);
    \draw[dashed] (3.75,1.45) -- (8.5,1.45);
    \node at (6.325,-1.25) {$B_H$};
    \draw[] (3.5,1.45) -- (3.5,-2.5);
    \draw[] (3.5,-2.5) -- (8.75,-2.5);
    \draw[] (8.75,-2.5) -- (8.75,1.45);
    \node at (6.4,2.95) {$B_{G'}$};
    \draw[] (3.5,1.45) -- (4.55,3.75);
    \draw[] (4.55,3.75) -- (4.55, 4.25);
    \draw[] (4.55, 4.25) -- (8.75, 4.25);
    \draw[] (8.75, 4.25) -- (8.75, 3.75);
    \draw[] (8.75, 3.75) -- (8.75,1.45);
    
    \end{tikzpicture}\\
    $G$ & $G'$\\
\end{tabular}

\caption{}
\label{Gdef}
\end{figure}
(see Figure \ref{Gdef}).
Thus, $c(G')= c(H) + (r - c(H)) = r$.
So, $G'$ is also a disjoint union of cliques and, by Observation \ref{indnumcomponents}, $\alpha(G')=r$.
Each vertex in $V(G-S)$ must be adjacent to at least one vertex in $S$, else there is an independent set in $G$ with more than $r$ vertices. Also, by construction of $G'$, each vertex in $H$ has exactly $1$ edge to the vertices in $G'-H$. 
So, if $E(v,W)$ denotes the set of edges between $v$ and $W$, then
\begin{equation*}
\begin{split}
\vert E(G) \vert & = \vert E(G-S) \vert + \sum_{v \in V(G-S)} \vert E(v,S) \vert\\
& \geq \vert E(H) \vert + \sum_{v \in V(G-S)} 1 \\
& = \vert E(H) \vert + \sum_{v \in V(H)} \vert E(v, V(G'-H)) \vert \\
& = \vert E(G') \vert.\\
\end{split}
\end{equation*}
That is to say,
\begin{equation} \label{edgesGG'}
\vert E(G) \vert \geq  \vert E(G') \vert.
\end{equation}
We now show $G'$ is an $(n,k,t,r)$-graph. 
First, let $A_{H}$ be the subgraph of $H$ consisting of connected components (necessarily cliques) with strictly fewer than $t-1$ vertices and $B_{H}$ be the subgraph of $H$ consisting of components
with at least $t-1$ vertices.
Lemma \ref{l.nktrelation} (b) applied to 
the $(n-r,k-r,t-1)$-graph $H$ 
gives
\begin{equation} \label{Hfact}
    k-r > (t-2)c(B_H) +\vert V(A_H) \vert.
\end{equation}
Now let $A_{G'}$ be the subgraph of $G'$ consisting of connected components (by construction of $G'$, necessarily cliques) with strictly fewer than $t$ vertices and $B_{G'}$ be the subgraph of $G'$ consisting of components (necessarily cliques)
with at least $t$ vertices.
Then:
\begin{itemize}
    \item $c(B_{G'})=c(B_H)$, as every clique of size $\geq t-1$ in $H$ has become a clique of size $\geq t$ in $G'$ (by construction of $G'$); and
    \item $\vert V(A_{G'}) \vert = \vert V(A_H) \vert + r - c(B_H)$, as every clique of size $< t-1$ in $H$ has become a clique of size $<t$ in $G'$, and only $c(B_H)$ of the $r$ vertices in $G'-H$ are added to the larger $B_{H}$ components.
\end{itemize}

Let $X$ be a subset of $V(G')$ with $\vert X \vert = k$.
By definition of $G'$, 
 at most $\vert V(A_{G'}) \vert = \vert V(A_H) \vert + r - c(B_H)$ vertices in $X$ are in components with fewer than $t$ vertices.
Thus, by the pigeonhole principle, some component of $B_{G'}$ contains at least $\frac{k- \vert V(A_{G'}) \vert}{c(B_{G'})}$ vertices from $X$.
Plus, we can lower bound
\begin{equation*}
\begin{aligned}
    \frac{k- \vert V(A_{G'}) \vert}{c(B_{G'})} & = \frac{ k-(\vert V(A_H) \vert + r - c(B_H))}{c(B_H)} & \\
    & = \frac{ k-\vert V(A_H) \vert -r}{c(B_H)} +1 & \\
    & > (t-2) + 1 & \text{(By \eqref{Hfact})}\\
    & =t-1. & \\
\end{aligned}
\end{equation*}
Thus, $G'[X]$ contains a $K_t$.
This proves $G'$ is an $(n,k,t,r)$-graph.
So, we have found an $(n,k,t,r)$-graph $G'$ which is a disjoint union of cliques and with $\vert E(G) \vert \geq \vert E(G') \vert$
(as needed for the Weak $(n,k,t)$-conjecture).

Now, suppose $G$ is also a minimum $(n,k,t,r)$-graph.
So the inequality in \eqref{edgesGG'} is actually an equality.
Thus, each vertex $v \in G-S$ is adjacent to exactly one vertex in $S$.
Also,  $\vert E(G-S) \vert = \vert E(H) \vert$, so $G-S$ is a minimum $(n-r, k-r, t-1)$-graph with independence number at most $r$ and is therefore also a disjoint union of cliques by the induction hypothesis.

We show $G$ must also be a disjoint union of cliques. This is implied by the following two additional facts.

\begin{itemize}
\item[$(\diamondsuit)$] If $u$,$v$ are two vertices in different components of $G-S$, then they cannot be adjacent to the same vertex $w \in S$.
\item[$(\heartsuit)$] If $u$,$v$ are two vertices in the same components of $G-S$, then they are adjacent to the same vertex in $S$.
\end{itemize}
 
To prove $(\diamondsuit)$, suppose for a contradiction that $u$ and $v$ are two vertices in the different components of $G-S$ that are adjacent to the same vertex $w \in S$.
Recall that $u$ and $v$ are not adjacent to any other vertices in $S$.
So, $S \setminus \{w\} \cup \{u,v\}$ is an independent set in $G$ with $r+1$ vertices. This is a contradiction.

We now prove $(\heartsuit)$.
Similarly to before, define $A_{G-S}$ to be the subgraph of $G-S$ consisting of connected components with strictly fewer than $t-1$ vertices and $B_{G-S}$ the subgraph of $G-S$ consisting of components
with at least $t-1$ vertices.
For a contradiction, suppose $u$ and $v$ are two vertices in the same component, $F$, of $G-S$ and they are adjacent to two distinct vertices in $S$.
 Because of $(\diamondsuit)$ and because the vertices in component $F$ are adjacent to at least $2$ distinct vertices in $S$, $c(G-S) < |S|=r$.
This leads to two important facts:
\begin{enumerate}[(i)]
\item Because $G-S$ is a disjoint union of cliques, by Observation \ref{indnumcomponents}, $\alpha(G-S)<r$.
\item $A_{G-S}$ contains only isolated vertices (else choose a vertex in $A_{G-S}$ with at least one incident edge and delete all such. This is still an $(n-r, k-r, t-1)$-graph with independence number at most $r$ by (i), and strictly fewer edges than $G-S$, contradicting minimality). In particular, as $|F| \geq 2$, $F$ must be a component of $B_{G-S}$.
\end{enumerate}

Let $Y$ be a set of $(t-2)(c(B_{G-S})-1)$ vertices, containing exactly $t-2$ vertices from each connected component of $B_{G-S}-V(F)$.
Let $Z$ be a set of $t-3$ vertices in $V(F)-\{u,v\}$ (this is possible because  of (ii)).
Define
\begin{equation*}
X'= S \sqcup V(A_{G-S}) \sqcup Y \sqcup Z \sqcup \{u,v\}.
\end{equation*}
We now show $\vert X' \vert =k$.
Consider two cases.
Suppose $k-r > t-1$.
Because $G-S$ is an $(n-r, k-r, t-1)$-graph,
using (i) 
and
Lemma \ref{notkminus1} shows $G-S$ is not an $(n- r, k-r-1, t-1)$-graph.
So
Lemma \ref{l.nktrelation} (c) gives
\begin{equation} \label{largestset}
    k-r-1
    =
    \vert V(A_{G-S}) \vert +(t-2)c(B_{G-S}).
\end{equation}
Now, suppose $k-r=t-1$.
Then, $G-S=K_{n-r}$, so $\vert V(A_{G-S}) \vert =0$ and $c(B_{G-S})=1$. 
Therefore, Equation \eqref{largestset} holds for all $k-r \geq t-1$.
Thus, in either case, because the union in the definition of $X'$ is necessarily a disjoint union,
\begin{equation*}
\begin{array}{llr}
    \vert X' \vert & \multicolumn{2}{l}{= r + \vert V(A_{G-S}) \vert + (t-2)(c(B_{G-S})-1) + (t-3) + 2}\\
    & = r + (k-r-1) + 1 & \text{ (By \eqref{largestset})}\\
    & = k. & \\
\end{array}
\end{equation*}

Any clique of size $t$ in $G[X']$ contains at most $1$ vertex from $S$ since $S$ is an independent set.
Thus, if $G[X']$ contains a $K_t$, then at least $t-1$ vertices in $X'$ must be in the same connected component of $G[V(A_{G-S}) \cup Y \cup Z \cup \{u,v\}]$.
The set $X'$ does not contain $t-1$ vertices from the same connected components in $A_{G-S}$ or $Y$.
The set $X'$ contains $t-1$ vertices from $V(F)$ (namely, $Z \cup \{u,v\}$), but because each vertex in $V(F)$ is adjacent to exactly $1$ vertex in $S$ and $u$ and $v$ are adjacent to two distinct vertices in $S$, $u$ and $v$ are not in a clique on $t$ vertices in $G[X']$.
Thus, $G[X']$ does not contain a clique on $t$ vertices, contradicting $G$ being an $(n,k,t)$-graph.
\end{proof}

Theorem \ref{mainthm} implies Theorem \ref{strongthm} proving the $(n,k,t)$-conjectures.
Indeed, a minimum $(n,k,t)$-graph is one with the minimum number of edges among all minimum $(n,k,t,r)$-graphs for $1 \leq r < k-t+2$.

Note, even if we relax the definition of an $(n,k,t,r)$-graph and let $n$, $k$, and $t$ be any positive integers, Theorem \ref{mainthm} still holds.
If $k > n$ then there does not exist a set of $k$ vertices, so all graphs on $n$ vertices are $(n,k,t,r)$-graphs.
Thus, the unique minimum $(n,k,t,r)$-graph is $\cTr{n}{r}$ by Tur\'{a}n's Theorem.
Also, if $n \geq k$ and $t>k$, then an induced subgraph on $k$ vertices cannot contain a clique on $t$ vertices, so no graphs are $(n,k,t,r)$-graphs.
So, the theorem holds vacuously.

In light of Observation \ref{uniqueindnum}, one might be led to believe for every positive integer $r$, there exists a unique minimum $(n,k,t,r)$-graph.
However, this is not the case.

\begin{observation}
There exist $n,k,t,$ and $r$ for which 
the minimum $(n,k,t,r)$-graph is not unique.
\end{observation}

For example, $2K_2+K_5$ and $K_1+2K_4$ are both minimum $(9,8,4,3)$-graphs.
These can each be formed as described in the proof of Theorem \ref{mainthm} by letting $G-S$ be the minimum $(6,5,3)$-graphs, $2K_1+K_4$ or $K_3 + K_3$, respectively.
However, by Observation \ref{uniqueindnum} each minimum $(n,k,t)$-graph has a unique independence number.
In this example, by Theorem \ref{mincliquegraph} and Theorem \ref{strongthm}, the minimum $(9,8,4)$-graph is $4K_1 + K_5$ and this is the unique minimum $(9,8,4,5)$-graph.

\section{Future Directions}

Our main result and \cite{REU} together solve the extremal problem of finding the minimum number of edges in an $(n,k,t)$-graph.
Viewing edges as cliques on $2$ vertices leads to one possible generalization of this problem.

\begin{question} \label{futurescliques}
Let $s$ be a positive integer. What is the minimum number of cliques on $s$ vertices in an $(n,k,t)$-graph?
\end{question}

A logical first step may be to ask:
What is the minimum number of cliques on $s$ vertices in an $(n,k,2)$-graph?
Recall, when $s=2$, this is equivalent to Tur\'{a}n's Theorem.
For general $s$, it turns out this is a special case of a question asked by Erd\H{o}s \cite{Erdos}.
He conjectured that a disjoint union of cliques would always be best.
However, Nikiforov \cite{nik} disproved this conjecture, observing that a clique-blowup of $C_5$ has independence number $2$ (so is an $(n,3,2)$-graph), but has fewer cliques on $4$ vertices than $2K_{\frac{n}{2}}$---the graph which has fewest $K_4$'s among all $(n,3,2)$-graphs
that are a disjoint union of cliques.
Das et al. \cite{das} and Pikhurko and Vaughan \cite{pik} independently found the minimum number of cliques on $4$ vertices in an $(n,3,2)$-graph for $n$ sufficiently large and Pikhurko and Vaughan \cite{pik} found the minimum number of cliques on $5$ vertices, cliques on $6$ vertices, and cliques on $7$ vertices in an $(n,3,2)$-graph for $n$ sufficiently large.
But their approaches are non-elementary and rely heavily on Razborov's flag algebra method.
A summary on related results can be found in Razborov's survey \cite{razborov}.

In \cite{noble2017application}, Noble et al. showed for $n \geq k \geq t \geq 3$ and $k \leq 2t-2$ every $(n,k,t)$-graph must contain a $K_{n-k+t}$.
Because $(k-t) K_1 + K_{n-k+t}$ is an $(n,k,t)$-graph, this shows for $n \geq k \geq t \geq 3$ and $k \leq 2t-2$ the minimum number of $s$-cliques an $(n,k,t)$-graph must contain is ${n-k+t \choose s}$ 
(noting ${n-k+t \choose s}=0$ if $s > n-k+t$).

Thus, the smallest interesting open question of this form is as follows:
\begin{question}
What is the minimum possible number of triangles in an $(n,5,3)$-graph?
\end{question}

We can also consider the other end of the spectrum.
What is the minimum number of cliques on $t$ vertices in an $(n,k,t)$-graph?
Also, for what values of $n$, $k$, and $t$ does an $(n,k,t)$-graph necessarily contain at least one clique of $t+1$ vertices?

Similar to Question \ref{futurescliques}, one could ask: What is the minimum number $n$ that an $(n,k,t)$-graph must contain a clique on $s$ vertices?
However, if $t=2$, this is exactly the Ramsey number $R(k,s)$.
Thus, this question seems very hard in general.
%

One other common direction to take known results in extremal graph theory where extremal structures have been identified, is to ask whether every \emph{near-optimal} example can be modified to make the optimal example using relatively few edits. For example, a result of F\"{u}redi \cite{furedi2015proof}, when complemented, gives the following strengthening of Tur\'{a}n's Theorem:

\begin{theorem}\label{furedi}
Let $G$ be an $n$-vertex graph without an independent set of size $r+1$ with $|E(G)| \leq |E(\cTr{n}{r})|+q$. Then, upon adding at most $q$ additional edges, $G$ contains a spanning vertex-disjoint union of at most $r$ cliques.
\end{theorem}

Fixing an independence number $r$ and writing $G'$ for a minimum $(n,k,t,r)$-graph
(as in Theorem \ref{mainthm}), this leads us to the following question:
\begin{question}
Suppose $G$ is an $(n,k,t)$-graph with independence number $r$, satisfying $|E(G)| \leq |E(G')|+q$. Does there exist a function $f(q)$ such that only $f(q)$ edges need adding to make $G$ contain a spanning disjoint union of $\leq r$ cliques?
\end{question}

One would need the function $f(q)$ to be independent of $n$ in order for the above to be interesting. Does $f$ need to depend on $k$? Or on $t$? Theorem \ref{furedi} suggests that the answer to the first question may actually be no.

Finally, in light of the constructive nature of the proof of Theorem \ref{mainthm}, one may ask whether every \emph{inclusion-minimal} $(n,k,t)$-graph (that is, an $(n,k,t)$-graph which is no longer $(n,k,t)$ upon the removal of any edge) is necessarily a disjoint union of cliques. But, there are counterexamples to this even in the original setting of Tur\'an's theorem ($t=2$). For example, $C_5$ is an inclusion-minimal $(5,3,2)$-graph, but contains neither of the inclusion-minimal $(5,3,2)$-graphs which are disjoint unions of cliques ($K_2+K_3$ and $K_1+K_4$).

Since there are inclusion-minimal $(n,k,t)$-graphs which are not disjoint unions of cliques, this suggests the saturation problem for $(n,k,t)$-graphs is interesting:

\begin{question}
Among all inclusion-minimal $(n,k,t)$-graphs, what is the maximum possible number of edges?
\end{question}
The classical saturation result of Zykov \cite{zykov1949some} and Erd\H{o}s-Hajnal-Moon \cite{erdos1964problem} (when stated for graph complements) says that when $t=2$, the answer is given by $(k-2)K_1+K_{n-k+2}$.
The example above shows there are instances of inclusion-minimal $(n,k,t)$-graphs that are not disjoint unions of cliques, but $K_1+K_4$ still has more edges than $C_5$. In general, we conjecture that the maximum is still always attained by a disjoint union of cliques.


\section*{Acknowledgements}
We are indebted to Peter Johnson for introducing us to this problem and suggestions improving an earlier version of the script, and to Paul Horn for helpful discussions.
The first author would also like to thank Alex Stevens, Andrew Owens, Calum Buchanan, and the rest of the Masamu Graph Theory Research Group for thoughtful conversations about the topic.
A portion of this work was completed while the first author was a student at Auburn University.
Finally, we would like to thank the referees for their insightful comments and suggestions.


\end{document}